\documentclass[11pt]{amsart}
 \usepackage[dvips]{epsfig}
 \usepackage{comment}
 \usepackage{enumerate}
 \usepackage{amsgen, amstext,amsbsy,amsopn, amsthm, amsfonts,amssymb,amscd,amsmat
 h,euscript,enumerate,url,verbatim,calc,xypic, mathtools}

 \usepackage{latexsym}
 \usepackage{graphics}
 \usepackage{color}
\newcommand{\Aa}{\mathcal A}

\newcommand{\proset}{\,\mathrel{\lower 4pt\hbox{$\scriptscriptstyle/$}
\mkern -14mu\subseteq }\,} 
 
 \newtheorem{theorem}{Theorem}[section]
  \newtheorem{corollary}[theorem]{Corollary}
 \newtheorem{lemma}[theorem]{Lemma}
 \newtheorem{proposition}[theorem]{Proposition}

\newtheorem{remark}[theorem]{Remark}

 \newtheorem{example}[theorem]{Example}

\numberwithin{equation}{section}

\def\ker{\operatorname{ker}}

\def\Br{\operatorname{Br}}
\usepackage{amsmath}
 \makeatother 
 
\begin{document}
 \date{\today}
 
 \title{On the vanishing of Twisted negative K-theory and homotopy invariance}

 \author{ Vivek Sadhu}

 \address{Department of Mathematics, Indian Institute of Science Education and Research Bhopal, Bhopal Bypass Road, Bhauri, Bhopal-462066, Madhya Pradesh, India}
 \email{vsadhu@iiserb.ac.in, viveksadhu@gmail.com}
 

\thanks{Author was supported by NBHM Research Grant NBHM/MAT/2024-2025/102}

\begin{abstract}
 In this article, we revisit  Weibel's conjecture for twisted $K$-theory. We also examine the vanishing of twisted negative $K$-groups for Pr\"{u}fer domains. Furthermore, we observe that the homotopy invariance of twisted $K$-theory holds for (finite-dimensional) Pr\"{u}fer domains. 
\end{abstract}

 \maketitle
 
 \tableofcontents
 \section{Introduction}
 It is well known that for a regular noetherian scheme $X,$ the homotopy invariance of $K$-theory holds (i.e., the natural map $K_{n}(X)\to K_{n}(X \times \mathbb{A}^{r})$ is an isomorphism for all $r\geq 0$ and $n\in \mathbb{Z}$) and $K_{-n}(X)=0$ for all $n>0.$ This is not true for non-regular schemes in general. Therefore, it has been an interesting question to investigate certain classes of schemes for which homotopy invariance of algebraic $K$-theory holds and negative $K$-groups vanishes. In this direction, Weibel's conjectured  in \cite{Wei80} that for a $d$-dimensional Noetherian scheme $X,$ the following should hold:
 \begin{enumerate}
  \item $K_{-n}(X)=0$ for $n>d;$
  \item $K_{-n}(X)\cong K_{-n}(X \times \mathbb{A}^{r})$ for $n\geq d$ and $r\geq 0.$
 \end{enumerate}
This conjecture was first proven for varieties over a field (see \cite{CHSW}, \cite{GH} and \cite{Krishna}). For a finite-dimensional quasi-excellent Noetherian scheme, Kelly showed in \cite{SK} that the negative $K$-groups vanish (up to torsion) after dimension. In 2018, Kerz-Strunk-Tamme ultimately settled Weibel's conjecture (see Theorem B of \cite{KST}). A relative version of Weibel's conjecture is discussed in \cite{VS}.

In this article, we are mainly interested in similar types of questions (i.e., homotopy invariance and vanishing of negative $K$-groups) in the context of the twisted $K$-theory. Given an Azumaya algebra $\mathcal{A}$ over a scheme $S,$ one can define twisted $K$-group $K_{n}^{\mathcal{A}}(S)$ for $n\in \mathbb{Z}$ (see section \ref{def}). It is natural to ask Weibel's conjecture for $K_{n}^{\mathcal{A}}(S).$ In \cite{stap}, J. Stapleton discussed Weibel's conjecture for $K_{n}^{\mathcal{A}}(S)$ and proved the first part, i.e., vanishing of twisted negative $K$-groups (see Corollary 4.2 of \cite{stap}). The second part of this conjecture has also been discussed in the same paper except the boundary case, i.e., $n=d$ (see Theorem 4.3 of \cite{stap}). In section \ref{wei revisit}, we revisit Weibel's conjecture for twisted $K$-theory and give proof that also takes care of the boundary case. Here is our result (see Theorem \ref{twisted version of weibel's conj}):
 \begin{theorem}
    Let $S$ be a Noetherian scheme of dimension $d.$ Let  $\mathcal{A}$ be an Azumaya algebra of rank $q^2$ over $S.$ Then 
 \begin{enumerate}
  \item $K_{-n}^{\Aa}(S)=0$ for $n>d;$
  \item $S$ is $K_{-n}^{\Aa}$-regular for $n\geq d$, i.e., the natural map $K_{-n}^{\Aa}(S) \to K_{-n}^{\Aa}(S \times \mathbb{A}^{r})$ is an isomorphism for $n\geq d$ and $r\geq 0.$
 \end{enumerate}
   \end{theorem}
  
  A subring $V$ of a field $K$ is said to be {\it valuation ring} if for each $0\neq a\in K,$ either $a\in V$ or $a^{-1}\in V.$ We say that an integral domain $R$ is a {\it Pr\"{u}fer} domain if it is locally a valuation domain, i.e., $R_{\mathfrak{p}}$ is a valuation domain for all prime ideals $\mathfrak{p}$ of $R.$ 
   In \cite{KM}, Kelly and Morrow observed that algebraic $K$-theory is homotopy invariant and  negative $K$-groups vanishes for valuation rings (see Theorem 3.3 of \cite{KM}). Later, Banerjee and Sadhu in \cite{BS} extended the above mentioned results for Pr\"{u}fer domains (see Theorem 1.1 of \cite{BS}). In section \ref{twisted val}, we investigate the same for twisted $K$-groups. More precisely, we show (see  Example \ref{example of rings} and Corollary \ref{cor of gen case}):
   \begin{theorem}
    Let $\mathcal{A}$ be an Azumaya algebra of rank $q^2$ over a ring $R$ and $SB(\Aa)$ be the associated Severi Brauer variety. Assume that $R$ is a Pr\"{u}fer domain with finite krull dimension. Then
  \begin{enumerate}
  
  \item $K_{-n}^{\Aa}(R)=0$ for $n> \dim (SB(\Aa));$ 
  
   \item the natural map $K_{n}^{\Aa}(R) \to K_{n}^{\Aa}(R[t_1, t_2, \dots, t_r])$ is an isomorphism for all $n\in \mathbb{Z}$ and $r\geq 0.$

  \end{enumerate}
   \end{theorem}

   By Morita equivalence, for a ring $R$ and $n\in \mathbb{Z}$, $K_{n}(R)\cong K_{n}^{\Aa}(R)$  in the case when $\Aa$ is a matrix algebra over $R$. This isomorphism may not hold for all Azumaya algebras. In section \ref{k vs KA}, we examine the relationship between $K_{n}(R)$ and $K_{n}^{\Aa}(R),$ assuming $R$ is a valuation ring of characteristic $p.$ We show that there is an injection from  $K_{n}(R)$ to $K_{n}^{\Aa}(R)$ for all $n\geq 0$ provided the rank of $\Aa$ is $p^2$ (see Theorem \ref{injects}). 
   
   {\bf Acknowledgements:} The author would like to thank Charles Weibel for fruitful email exchanges. He would also like to thank the referee for
valuable comments and suggestions.
   
 \section{Twisted $K$-theory}\label{def}

 Let $A$ be an algebra (not necessarily commutative) over a commutative local ring $R.$ The opposite algebra $A^{op}$ of $A$ is the algebra $A$ with multiplication reversed. We say that $A$ is an {\it Azumaya algebra} over $R$ if it is free $R$-module of finite rank and the map $A \otimes_{R} A^{op} \to End_{R}(A), a\otimes a^{'}\mapsto (x\mapsto axa^{'})$ is an isomorphism. For example, the matrix algbera $M_{n}(R)$ is an Azumaya algebra over $R.$ Let $X$ be a scheme. An $\mathcal{O}_{X}$-algebra $\Aa$ is said to be an {\it Azumaya algebra} over $X$ if it is coherent, locally free  as an $\mathcal{O}_{X}$-module and  $\Aa_{x}$ is an Azumaya algebra over $\mathcal{O}_{X,x}$ for any point $x\in X.$ Equivalently, $\Aa$ is \'etale locally isomorphic to $M_{n}(\mathcal{O}_{X})$ for some $n.$ For details, see \cite{Milne}.
 
 \subsection{Twisted $K$-groups}
 Let $\mathcal{A}$ be an Azumaya algebra over a scheme $S.$ Let ${\bf{Vect}^{\mathcal{A}}}(S)$ denote the category of vector bundles on $S$ that are left modules for $\mathcal{A}.$ The category ${\bf{Vect}^{\mathcal{A}}}(S)$ is exact. The twisted $K$-theory space is defined by $K^{\mathcal{A}}(S):= K({\bf{Vect}^{\mathcal{A}}}(S)).$ For $n\geq 0,$ the $n$-th twisted $K$-group $K_{n}^{\mathcal{A}}(S)$ is defined as $\pi_{n}(K({\bf{Vect}^{\mathcal{A}}}(S)).$ 

Write $S[t]$ for $S\times_{\mathbb{Z}} \mathbb{Z}[t]$ and $S[t, t^{-1}]$ for $S\times_{\mathbb{Z}} \mathbb{Z}[t, t^{-1}].$ Since the projection map $p: S[t]\to S$ is flat, it induces an exact functor $p^{*}: K^{\mathcal{A}}(S) \to K^{p^{*}\mathcal{A}}(S[t]).$ Thus we have maps between twisted $K$-groups $K_{n}^{\mathcal{A}}(S) \to K_{n}^{p^{*}\mathcal{A}}(S[t]).$ By abuse of notation, we write $K_{n}^{\mathcal{A}}(S[t])$ instead of $K_{n}^{p^{*}\mathcal{A}}(S[t]).$ Similarly, we also have maps between $K_{n}^{\mathcal{A}}(S) \to K_{n}^{\mathcal{A}}(S[t, t^{-1}]).$ Following Bass (see chapter XII of \cite{b}), the twisted negative $K$-group $K_{-1}^{\mathcal{A}}(S)$ is defined as
$$Coker[K_{0}^{\mathcal{A}}(S[t])\times K_{0}^{\mathcal{A}}(S[t^{-1}]) \stackrel{\pm}\to K_{0}^{\mathcal{A}}(S[t, t^{-1}])].$$ By iterating, we have 

 $$K_{-n}^{\mathcal{A}}(S):= Coker[K_{-n+1}^{\mathcal{A}}(S[t])\times K_{-n+1}^{\mathcal{A}}(S[t^{-1}]) \stackrel{\pm}\to K_{-n+1}^{\mathcal{A}}(S[t, t^{-1}])].$$ There is a split exact sequence for $n\in \mathbb{Z}$ (see section 3 of \cite{stap})
 \begin{equation}\label{fund seq}
  0 \to K_{n}^{\mathcal{A}}(S) \stackrel{\Delta}\to K_{n}^{\mathcal{A}}(S[t])\times K_{n}^{\mathcal{A}}(S[t^{-1}]) \stackrel{\pm}\to K_{n}^{\mathcal{A}}(S[t, t^{-1}])\to K_{n-1}^{\mathcal{A}}(S)\to 0,
 \end{equation}
where $\Delta(a)=(a, a)$ and $\pm(a,b)=a-b.$

\subsection{Quillen's generalized projective bundle formula}
It is well-known that there is a natural bijection of sets
\tiny
$$\{ {\rm Severi-Brauer~ varieties~  of~ relative ~dimension~ (q-1) ~over~ S} \}\longleftrightarrow \{ {\rm Azumaya~ algebras ~over~S ~of~ rank~ q^2 } \}.$$
\normalsize
Let $\mathcal{A}$ be an Azumaya algebra of rank $q^2$ over a scheme $S.$ One can associate a Severi-Brauer variety $SB(\Aa)$ of relative dimension $q-1$ over $S.$ The structure morphism $SB(\mathcal{A}) \to S$ is always smooth and projective. Quillen's generalized projective bundle formula state that there is a natural isomorphism for each $n\geq 0$(see Theorem 4.1 of \cite{DQ} or V.1.6.6 of \cite{wei 1}),
\begin{equation}\label{Quillen decomposition}
 K_{n}(SB(\Aa))\cong \bigoplus_{i=0}^{q-1} K_{n}^{\Aa ^{\otimes i}}(S).
\end{equation}
We consider the following commutative diagram
 \tiny
 $$\begin{CD} 0 @. 0  @. 0 @.  \\
   @VVV   @VVV @VVV \\
   K_{0}(S) @> injects>> K_{0}(S[t])\times K_{0}(S[t^{-1}]) @>>> K_{0}(S[t, t^{-1}) @>>> K_{-1}(S) @>>> 0\\
   @VVV      @VVV     @VVV @VVV \\
    K_{0}(SB(\Aa)) @> injects>> K_{0}(SB(\Aa)[t])\times K_{0}(SB(\Aa)[t^{-1}]) @>>> K_{0}(SB(\Aa)[t, t^{-1}) @>>> K_{-1}(SB(\Aa)) @>>> 0 \\
    @VVV   @VVV @VVV @VVV\\
   \bigoplus_{i=1}^{q-1}K_{0}^{\Aa ^{\otimes i}}(S) @> injects >> \bigoplus_{i=1}^{q-1}K_{0}^{\Aa ^{\otimes i}}(S[t])\times \bigoplus_{i=1}^{q-1}K_{0}^{\Aa ^{\otimes i}}(S[t^{-1}) @>>> \bigoplus_{i=1}^{q-1}K_{0}^{\Aa ^{\otimes i}}(S[t, t^{-1}]) @>>> \bigoplus_{i=1}^{q-1}K_{-1}^{\Aa ^{\otimes i}}(S) @>>> 0 \\
   @VVV @VVV @VVV \\
   0 @. 0  @. 0 @. 
 \end{CD}.$$
\normalsize
By the fundamental theorem of $K$-theory and (\ref{fund seq}), the rows are  split exact. The first three columns are also split exact by (\ref{Quillen decomposition}).  Finally, a diagram chase gives a natural isomorphism
\begin{equation}\label{Quillen decomposition -1}
 K_{-1}(SB(\Aa))\cong \bigoplus_{i=0}^{q-1} K_{-1}^{\Aa ^{\otimes i}}(S).
\end{equation}
By iterating, we conclude that for each  $n\in \mathbb{Z},$ there is a natural isomorphism
\begin{equation}\label{Quillen decomposition for all n}
 K_{n}(SB(\Aa))\cong \bigoplus_{i=0}^{q-1} K_{n}^{\Aa ^{\otimes i}}(S).
\end{equation}
\begin{proposition}\label{reg case}
 Let $\mathcal{A}$ be an Azumaya algebra of rank $q^2$ over a Noetherian regular scheme $S.$ Then $K_{n}^{\Aa}(S)=0$ for $n<0$ and $K_{n}^{\Aa}(S)\cong K_{n}^{\Aa}(S[t_1, \dots, t_r])$ for all $n$ and $r\geq0.$ 
\end{proposition}
\begin{proof}
 Since $S$ is a Noetherian regular scheme, so is $SB(\Aa).$ In this situation, we know $K_{n}(SB(\Aa))=0$ for $n<0$ and $K_{n}(SB(\Aa))\cong K_{n}(SB(\Aa)[t_1, \dots, t_r])$ for all $n$ and $r\geq 0.$ By (\ref{Quillen decomposition for all n}), we get the result.
\end{proof}

\subsection{Brauer groups vs Twisted $K$-theory} We say that two unital rings $A$ and $B$ (possibly non-commutative) are {\it Morita equivalent} if the categories ${\bf Mod}_{A}$  and ${\bf Mod}_{B}$ of right modules are equivalent. For example, a unital ring $R$ is Morita equivalent to $M_{n}(R)$ for $n\geq 0.$

 Two Azumaya algebras $\Aa$ and $\mathcal{B}$ over a commutative ring $R$ are Morita equivalent if and only if there exist finitely generated projective $R$-modules $P$ and $Q$ such that $\Aa \otimes_{R} End(P)\cong \mathcal{B}\otimes_{R} End(Q)$ (see Theorem 1.3.15 of \cite{cal}). However, this is not true for Azumaya algebras over scheme, for instance see Example 1.3.16 of \cite{cal}. If $R$ is a commutative local ring then  $\Aa$ and $\mathcal{B}$ are Morita equivalent if and only if $M_{n}(\Aa)\cong  M_{m}(\mathcal{B})$ for $n, m>0.$
The Brauer group $\Br(R)$ of a commutative ring $R$ consists of Morita equivalence classes of Azumaya algebras over $R$ (see \cite{Milne}). The group operation on $\Br(R)$ is $\otimes_{R}.$ An element of $\Br(R)$ is represented by a class $[\Aa],$ where $\Aa$ is an Azumaya algebra over $R.$ The inverse of $[\Aa]$ is given by $[\Aa^{op}].$

Let $R$ be a commutative ring with unity. For $n\in \mathbb{Z},$ we consider the set
 $$\mathcal{F}_{n}=\{K_{n}^{\Aa}(R)| [\Aa]\in \Br(R)\}.$$
  An equivalence relation $\sim$ on $\mathcal{F}_{n}$ is given by  $K_{n}^{\Aa}(R) \sim K_{n}^{\Aa^{'}}(R)$ if  $K_{n}^{\Aa}(R)\cong K_{n}^{\Aa^{'}}(R).$ Define $BK_{n}(R):= \mathcal{F}_{n}/\sim.$ An element of $BK_{n}(R)$ is represented by a class $(K_{n}^{\Aa}(R)).$

 \begin{lemma}
  For $n\in \mathbb{Z},$ $BK_{n}(R)$ is an abelian group with the operation
  $$(K_{n}^{\Aa}(R))\ast (K_{n}^{\Aa^{'}}(R))= (K_{n}^{\Aa \otimes_{R}\Aa^{'}}(R)).$$
 \end{lemma}
\begin{proof}
 If $\Aa$ and $\mathcal{B}$ both are Azumaya algebras over $R$ then $\Aa \otimes_{R} \mathcal{B}$ is also an Azumaya algebra over $R.$ Thus, $\ast$ is closed. Since $\otimes_{R}$ is associative and abelian, so is $\ast.$ We know that algebraic $K$-theory is Morita invariant, i.e., for all $n\in \mathbb{Z},$ $K_{n}(R)\cong K_{n}(S)$  whenever $R$ and $S$ are Morita equivalent. This implies that $(K_{n}(R))$ is the identity element. The inverse of $(K_{n}^{\Aa}(R))$ is given by $(K_{n}^{\Aa^{op}}(R)).$  
\end{proof}

 We define a map $\psi_{n}: \Br(R) \to BK_{n}(R), [\Aa] \mapsto (K_{n}^{\Aa}(R))$ for each $n\in \mathbb{Z}.$ Each $\psi_{n}$ is a well defined map because $K$-theory is Morita invariant. Moreover, one can check the following:
 
 \begin{proposition}
  For a commutative ring $R,$  there is a short exact sequence
  $$ 0\to \ker \psi_{n}\to \Br(R) \to BK_{n}(R) \to 0$$ of abelian group for each $n\in \mathbb{Z}.$ Moreover,
  $$\ker \psi_{n}= \{\Aa\in Az(R)| K_{n}^{\Aa}(R)\cong K_{n}(R)\}.$$
  \end{proposition}

 \begin{remark}{\rm
 \begin{enumerate}
  \item If $\Br(R)=0$ then there are no twisted $K$-groups. 
  \item If $R=\mathbb{R}$ then $\Br(\mathbb{R})=\mathbb{Z}/2\mathbb{Z}=\{\mathbb{R}, \mathbb{H}\}.$ We know $K_{1}^{\mathbb{H}}(\mathbb{R})\ncong K_{1}(\mathbb{R})$ (see Table VI.3.1.1 of \cite{wei 1}). In this case, $\ker \psi_{1}=0.$ By Proposition \ref{reg case}, $K_{n}^{\mathbb{H}}(\mathbb{R})=0$ for $n<0.$ So, $\ker \psi_{n}= \mathbb{Z}/2\mathbb{Z}$ for $n<0.$
 \end{enumerate}}
\end{remark}
 
\section{Twisted version of Weibel's Conjecture}\label{wei revisit}
Throughout, $\Aa$ is an Azumaya algebra of rank $q^2$ over a scheme $S$ and $SB(\mathcal{A})$ is the associated Severi-Brauer variety. We would like to understand the $K$-theory of the structure map $\rho: SB(\mathcal{A}) \to S.$ 

Let $f: X\to S$ be a map of schemes. 
Let $K(f)$ denote the homotopy fibre of $K(S) \to K(X).$ Here $K(X)$ denotes the Bass non-connective $K$-theory spectrum of a scheme $X.$ We have the associated long exact sequence
\begin{equation}\label{les}
 \dots \to K_{n}(f) \to K_n(S) \to K_{n}(X)\to K_{n-1}(f)\to K_{n-1}(S)\to \dots
\end{equation}

 Let $F$ be a functor from category of rings (or schemes) to abelian groups. Let $NF(X)= \ker [F(X\times \mathbb{A}^{1}) \to F(X)].$ There is a natural decomposition $F(X\times \mathbb{A}^{1})\cong F(X) \oplus NF(X).$ By iterating, one can define $N^{t}F(X).$ We have a natural decomposition $F (X\times \mathbb{A}^{r})\cong (1+N)^{r}F(X).$  We say that $X$ is $F$-regular if the natural map $F(X) \to F (X\times \mathbb{A}^{r})$ is an isomorphism for $r\geq 0.$ Equivalently, $N^{r}F(X)=0$ for $r>0.$
 
 By comparing the exact sequences (\ref{les}) for $f$ and $f \times \mathbb{A}^{1}:$ $X\times \mathbb{A}^{1} \to S \times \mathbb{A}^{1},$ a diagram chase gives a long exact sequence for $NK_{*}.$ By iterating, we also have a long exact sequence for $N^{r}K_{*}$ (some more details related to $NK_{*}$-groups can be found in section 3 of \cite{VS})
\begin{equation}\label{les2}
 \dots \to N^{r}K_{n}(f) \to N^{r}K_n(S) \to N^{r}K_{n}(X)\to N^{r}K_{n-1}(f)\to N^{r}K_{n-1}(S)\to \dots
\end{equation}
 
 The following result is due to Kerz-Strunk-Tamme.
 
\begin{theorem}\label{k-dimension conj}
 Let $X$ be a Noetherian scheme of dimension $d.$ Then 
 \begin{enumerate}
  \item $K_{-n}(X)=0$ for $n>d;$
  \item $X$ is $K_{-n}$-regular for $n\geq d$, i.e., the natural map $K_{-n}(X) \to K_{-n}(X \times \mathbb{A}^{r})$ is an isomorphism for $n\geq d$ and $r\geq 0.$
 \end{enumerate}

\end{theorem}

\begin{proof}
 See Theorem B of \cite{KST}.
\end{proof}

A relative version of the aforementioned theorem is as follows:
\begin{theorem}\label{smooth quasiprojective}
    Let $f: X \to S$ be a smooth, quasi-projective map of noetherian schemes with $S$ reduced. Assume that $\dim S= d.$ Then $K_{-n}(f)=0$ for $n>d+1$ and $f$ is $K_{-n}$-regular for $n>d,$ i.e., the natural map $K_{-n}(f) \to K_{-n}(f \times \mathbb{A}^{r})$ is an isomorphism for $n> d$ and $r\geq 0.$ Here $f \times \mathbb{A}^{r}$ denotes $X\times \mathbb{A}^{r} \to S \times \mathbb{A}^{r}.$
   \end{theorem}
   \begin{proof}
    See Theorem 3.8 of \cite{VS}.
   \end{proof}

   Let $\mathcal{K}^{\Aa}_{n, zar}$ denote the Zariski sheafification of the presheaf $U\mapsto K_{n}^{\Aa}(U).$ Similarly,  $\mathcal{NK}_{n, zar}^{\Aa}$ is the Zariski sheafification of the presheaf $U\mapsto NK_{n}^{\Aa}(U).$ More generally, one can define $\mathcal{N}^{r}\mathcal{K}_{n, zar}^{\Aa}$ for $r>0.$ 
   
   \begin{lemma}\label{reduced}
     Let $S$ be a Noetherian scheme of dimension $d.$ Let  $\mathcal{A}$ be an Azumaya algebra of rank $q^2$ over $S.$ Then $K_{-n}^{\Aa}(S)\cong K_{-n}^{\Aa}(S_{red})$ and $N^{r}K_{-n}^{\Aa}(S)\cong N^{r}K_{-n}^{\Aa}(S_{red})$ for $n\geq d$ and $r>0.$
   \end{lemma}
\begin{proof}
 Given a commutative ring $R,$ $K_{-n}^{\Aa}(R)\cong K_{-n}^{\Aa}(R_{red})$ for $n\geq 0$ (see Proposition 2.7 of \cite{stap}). Note $(R[t])_{red}=R_{red}[t].$ Thus,  $N^{r}K_{-n}^{\Aa}(R)\cong N^{r}K_{-n}^{\Aa}(R_{red})$ for $n\geq 0$ and $r>0.$ The rest of the argument is based on comparing Zariski descent spectral sequences for $S$ and $S_{red}$ (see Corollary 2.8 of \cite{stap} and Lemma 3.4 of \cite{VS}).
\end{proof}

   \begin{theorem}\label{twisted version of weibel's conj}
    Let $S$ be a Noetherian scheme of dimension $d.$ Let  $\mathcal{A}$ be an Azumaya algebra of rank $q^2$ over $S.$ Then 
 \begin{enumerate}
  \item $K_{-n}^{\Aa}(S)=0$ for $n>d;$
  \item $S$ is $K_{-n}^{\Aa}$-regular for $n\geq d$, i.e., the natural map $K_{-n}^{\Aa}(S) \to K_{-n}^{\Aa}(S \times \mathbb{A}^{r})$ is an isomorphism for $n\geq d$ and $r\geq 0.$
 \end{enumerate}
   \end{theorem}
\begin{proof}
We may assume that $S$ is reduced (see Lemma \ref{reduced}). Let $SB(\Aa)$ be the associated Severi-Brauer variety of relative dimension $q-1$ over $S.$ Note $\rho: SB(\Aa) \to S$ is a smooth, projective morphism (hence also finite type). Since $S$ is Noetherian, so is $SB(\Aa).$ Then 
$K_{-n}(\rho)=0$ for $n>d+1$ and $N^{r}K_{-n}(\rho)=0$ for all $r\geq 0$ and $n>d$ by Theorem \ref{smooth quasiprojective}. The sequence (\ref{les}) implies that $K_{-d-1}(S) \to K_{-d-1}(SB(\Aa))$ is surjective and $K_{-n}(S) \to K_{-n}(SB(\Aa))$ is an isomorphism for $n>d+1.$ Similarly, the sequence (\ref{les2}) implies that $N^{r}K_{-d}(S) \to N^{r}K_{-d}(SB(\Aa))$ is surjective and $N^{r}K_{-n}(S) \to N^{r}K_{-n}(SB(\Aa))$ is an isomorphism for $n>d.$ By Theorem \ref{k-dimension conj}, we get $K_{-n}(SB(\Aa))=0$ for $n> d$ and $N^{r}K_{-n}(SB(\Aa))=0$ for $n\geq d.$ The natural decomposition (\ref{Quillen decomposition for all n}) yields the result.
\end{proof}

 \section{Twisted $K$-theory of weakly regular stably coherent rings} \label{twisted val}

 Let $R$ be a commutative ring. A finitely generated $R$-module $M$ is called {\it coherent} if every finitely generated submodule of $M$ is finitely presented. The ring $R$ is {\it coherent} if it is a coherent module over itself, i.e., every finitely generated ideal of $R$ is finitely presented. The ring $R$ is said to be a {\it regular} ring if every finitely generated ideal of $R$ has finite projective dimension.
 
 Let $M$ be a $R$-module. The {\it weak dimension} of $M,$ denoted by $w.dim_{R}M$ is the least nonnegative integer $n,$ for which there is an exact sequence

$$ 0 \to F_n \to \dots \to F_1 \to F_0 \to M \to 0 $$ with $F_i$ flat over $R.$ The {\it weak dimension} of a ring $R,$ denoted $w.dim (R),$ is defined by  $w.dim (R)= Sup\{w.dim_{R}M| M ~{\rm is~ a~} R{\rm -module}\}.$ If $R$ is a coherent ring then $w.dim (R)= Sup\{Pd_{R}M| M ~{\rm is~ a~ finitely~ presented~}  R{\rm -module}\}$ (see Corollary 2.5.6 of \cite{Glaz}). Clearly, $Pd_{R}(M)\leq w.dim(R)$ for all finitely presented $R$-modules $M.$ We also have $w.dim(R_{\mathfrak{p}})\leq w.dim(R)$ for all $\mathfrak{p}\in Spec(R)$ (see Theorem 1.3.13 of \cite{Glaz}).

A coherent ring $R$ is called {\it weakly regular} if $R$ has finite flat (or weak) dimension. A ring $R$ is said to be {\it stably coherent} if every finitely presented $R$-algebra is coherent.

\begin{example}\label{example of rings}
 {\rm Here is a list of weakly regular stably coherent rings:
 \begin{enumerate}
 \item Noetherian regular local rings of finite krull dimension. In this case, global dimension coincides with weak dimension.
 \item Valuation rings (see Proposition 2.1 of \cite{AMM});
 \item Pr\"{u}fer domains (see Lemma 3.1 of \cite{BS} and P.25 of \cite{Glaz}).
\end{enumerate}}
\end{example}

 Let $\mathcal{K}_{n, zar}$ denote the Zariski sheafification of the presheaf $U\mapsto K_{n}(U).$ Similarly,  $\mathcal{NK}_{n, zar}$ is the Zariski sheafification of the presheaf $U\mapsto NK_{n}(U).$ More generally, one can define $\mathcal{N}^{r}\mathcal{K}_{n, zar}$ for $r>0.$

\begin{lemma}\label{red}
 Let $X$ be a scheme and $X\to Spec(R)$ be a  smooth map with $R$ weakly regular stably coherent. Then the Zariski sheaves on $X,$ $\mathcal{K}_{n, zar}=0$ for $n< 0$ and $\mathcal{N}^{r}\mathcal{K}_{n, zar}=0$ for $n \in \mathbb{Z},$  $r>0.$ 
 \end{lemma}
\begin{proof}
 Let $Spec(A)\hookrightarrow X$ be an affine open subset. Then $Spec(A) \to Spec(R)$ is smooth. Note $A$ is weakly regular stably coherent (see Corollary 2.3 of \cite{AMM}). Since any localization of stably coherent ring is stably coherent and $w.dim(A_{\mathfrak{p}})\leq w.dim(A)< \infty,$ $A_{\mathfrak{p}}$ is weakly regular stably coherent for all $\mathfrak{p}\in Spec(A).$ 
 Each stalk of $\mathcal{K}_{n, zar}$ and $\mathcal{N}^{r}\mathcal{K}_{n, zar}$ are $K_{n}( A_{\mathfrak{p}})$ and $N^{r}K_{n}(A_{\mathfrak{p}}),$ where $A_{\mathfrak{p}}$ is  weakly regular stably coherent. By Proposition 2.4 of \cite{AMM}, algebraic $K$-theory is homotopy invariant, and negative $K$-groups vanish for weakly regular stably coherent rings. Hence the assertion.   
\end{proof}

\begin{lemma}\label{qcqs gen}
  Let $X \to S$ be a projective morphism with $S$ quasi-compact and quasi separated. Then $X$ is also quasi-compact and quasi-separated.
\end{lemma}
\begin{proof}
 Projective morphisms are always quasi-compact and quasi separated morphism. Since $S$ quasi-compact and quasi separated scheme, so is $X.$ 
\end{proof}

\begin{theorem}\label{gen case}
 Let $X$ be a finite dimensional quasi-compact and quasi-separated scheme and  $X\to Spec(R)$ be a  smooth map with $R$ weakly regular stably coherent. Then
 
 \begin{enumerate}
  \item $K_{-n}(X)=0$ for $n>d$ and $H_{Zar}^{d}(X, \mathbb{Z})\cong K_{-d}(X),$ where $d=\dim(X);$
  \item The natural map $K_{n}(X) \to K_{n}(X \times \mathbb{A}^{r})$ is an isomorphism for $n\in \mathbb{Z}$ and $r\geq 0.$
  \end{enumerate}
  \end{theorem}

\begin{proof}
(1)  The scheme $X$ has finite Krull dimension $d.$ We have a descent spectral sequence (see Theorem 4.1 of \cite{RP} and Remark 3.3.1 of \cite{CM})
 $$H_{Zar}^{p}(X, \mathcal{K}_{n, Zar})\Longrightarrow K_{n-p}(X).$$ 
 Here $\mathcal{K}_{n, Zar}$ is the Zariski sheaf on $X.$ By Corollary 4.6 of \cite{CS}, $X_{Zar}$ has cohomological dimension at most $d=\dim(X).$ Moreover, $\mathcal{K}_{n, Zar}=0$ for $n< 0$ (see Lemma \ref{red}). This implies that $K_{-n}(X)=0$ for $n> d$ and $H_{Zar}^{d}(X, \mathbb{Z})=K_{-d}(X).$  
 
 (2) Consider the Zariski sheaf $\mathcal{N}^{r}\mathcal{K}_{n, Zar}$ on $X.$  We have $\mathcal{N}^{r}\mathcal{K}_{n, Zar}=0$ for $n \in \mathbb{Z},$ $r>0$ (see Lemma \ref{red}). The following descent spectral sequence 
 
 $$H_{Zar}^{p}(X, \mathcal{N}^{r}\mathcal{K}_{n, Zar})\Longrightarrow N^{r}K_{n-p}(X)$$ implies $N^{r}K_{n}(X)=0$ for $n \in \mathbb{Z},$ $r>0.$ 
\end{proof}

\begin{corollary}\label{cor of gen case}
 Let $R$ be a finite dimensional weakly regular stably coherent ring. Let $\Aa$ be an Azumaya algebra over $R$ of rank $q^2$ and $SB(\Aa)$ be the associated Severi Brauer variety. Then
  \begin{enumerate}
  
  \item $K_{-n}^{\Aa}(R)=0$ for $n> \dim (SB(\Aa)).$ 
  
   \item the natural map $K_{n}^{\Aa}(R) \to K_{n}^{\Aa}(R[t_1, t_2, \dots, t_r])$ is an isomorphism for all $n\in \mathbb{Z}$ and $r\geq 0.$

  \end{enumerate} 
\end{corollary}
\begin{proof}
The structure morphism $SB(\mathcal{A}) \to Spec(R)$ is smooth and projective (hence of finite type). The Severi Brauer variety $SB(\mathcal{A})$ has finite Krull dimension because $R$ is finite dimensional. By Lemma \ref{qcqs gen}, $SB(\mathcal{A})$ is a quasi-compact and quasi-separated scheme. The result now follows from Theorem \ref{gen case} and the decomposition (\ref{Quillen decomposition for all n}).
\end{proof}

\section{An observation}\label{k vs KA}
Let $\Aa$ and $\mathcal{B}$ be Azumaya algebras over a scheme $X.$ Assume that $\varphi: \mathcal{B} \to \Aa$ is an $\mathcal{O}_{X}$-algebra homomorphism and $\Aa$ is a flat $\mathcal{B}$-module. Then the functor $$-\otimes_{\mathcal{B}} {\Aa}: {\bf{Vect}^{\mathcal{B}}}(X)  \to {\bf{Vect}^{\mathcal{A}}}(X), P\mapsto P\otimes_ {\mathcal{B}} {\Aa}$$ is exact and it induces a group homomorphism $\varphi_{n}: K_{n}^{\mathcal{B}}(X) \to K_{n}^{\Aa}(X)$ for each $n\geq 0.$ 
 We also have a restriction functor $res_{\mathcal{B}}^{\Aa}: {\bf{Vect}^{\mathcal{A}}}(X) \to {\bf{Vect}^{\mathcal{B}}}(X),$ which is exact. It induces a group homomorphism $\phi_{n}: K_{n}^{\mathcal{A}}(X) \to K_{n}^{\mathcal{B}}(X)$ for each $n\geq 0.$
 
 If $\mathcal{B}=\mathcal{O}_{X}$ then $\Aa$ is a flat $\mathcal{O}_{X}$-module and $K_{n}^{\mathcal{O}_{X}}(X)=K_{n}(X)$. For $n \geq 0,$ we get group homomorphisms 
 $$\varphi_{n}: K_{n}(X) \to K_{n}^{\Aa}(X)$$ and
 
 $$\phi_{n}: K_{n}^{\Aa}(X) \to K_{n}(X).$$
 
 The composition  $\phi_{n} \varphi_{n}: K_{n}(X) \to K_{n}(X)$ is a map multiplication by $[\Aa]\in K_{0}(X).$
 \begin{theorem}\label{injects}
  Let $V$ be a valuation ring of characteristic $p>0.$ Let $\Aa$ be a Azumaya algebra over $V$ of rank $q^2,$ where $q=p^r$ for some $r \geq 1.$  Then the map $\varphi_{n}: K_{n}(V) \to K_{n}^{\Aa}(V)$ is injective for all $n\geq 0.$
 \end{theorem}

 \begin{proof}
    We have $[\Aa].\ker(\varphi_{n})=0$ for $n\geq 0$ (see the above discussion or Proposition 2 of \cite{HH}). Since $V$ is local, $\Aa$ is free over $V$ of rank $q^2.$ Thus, $q^2.\ker(\varphi_{n})=0.$ On the otherhand,  $K_{n}(V)$ is $p$-torsion free for $n\geq 0$ (see Theorem 1.1 of \cite{KM}). So, $\ker(\varphi_{n})$ is also $p$-torsion free for $n\geq 0$. This forces that $\ker(\varphi_{n})=0$ for $n\geq 0.$  Hence the assertion.
 \end{proof}

\end{document}